\documentclass[12pt,a4paper]{amsart}
\usepackage{amsmath,amssymb,amscd,amsthm,amsbsy,times}
\usepackage[initials,short-journals]{amsrefs}
\DeclareFontFamily{U}{rsfs}{\skewchar\font"7F}
\DeclareFontShape{U}{rsfs}{m}{n}{
	<-6> rsfs5
	<6-8> rsfs7
	<8-> rsfs10
	}{}
\DeclareMathAlphabet{\mathscr}{U}{rsfs}{m}{n}
\usepackage[shortlabels,inline]{enumitem}
\allowdisplaybreaks
\newcommand{\CF}{{\mathcal F}}

\newcommand{\C}{{\mathbb C}}

\newcommand{\R}{{\mathbb R}}
\newcommand{\Z}{{\mathbb Z}}

\newcommand{\DBott}{\mathrm{DBott}}
\newcommand{\DGV}{\mathrm{DGV}}
\newcommand{\FLK}{\mathrm{FLK}}
\DeclareMathOperator{\tr}{tr}
\newcommand{\gl}{\mathfrak{gl}}
\newcommand{\GL}{\mathrm{GL}}
\newcommand{\vG}{\varGamma}
\newcommand{\pdif}[2]{\dfrac{\partial#1}{\partial#2}}

\newcommand{\x}{\times}
\DeclareMathOperator{\disj}{\mathchoice{\coprod}{\textstyle\amalg}{\scriptstyle\amalg}{\scriptscriptstyle\amalg}}
 \newtheorem{theorem}{Theorem}[section]
 
 \newtheorem{proposition}[theorem]{Proposition}
 \newtheorem{lemma}[theorem]{Lemma}

\theoremstyle{definition}
 \newtheorem{definition}[theorem]{Definition}
 
 \newtheorem{example}[theorem]{Example}
\theoremstyle{remark}
 \newtheorem{remark}[theorem]{Remark}
\title{On a characteristic class associated\\ with deformations of foliations}
\author{Taro Asuke}
\address{Graduate School of Mathematical Sciences, University of Tokyo, 3-8-1 Komaba, Meguro-ku, Tokyo 153-8914, Japan}
\email{asuke@ms.u-tokyo.ac.jp}
\date{November 9, 2020}%\today
\begin{document}
\thispagestyle{plain}
\begin{abstract}
A characteristic class for deformations of foliations called the Fuks--Lodder--Kotschick class (FLK class for short) is studied.
It seems unknown if there is a real foliation with non-trivial FLK class.
In this article, we show some conditions to assure the triviality of the FLK class.
On the other hand, we show that the FLK class is easily to be non-trivial for transversely holomorphic foliations.
We present an example and give a construction which generalizes it.
\end{abstract}
\maketitle
\setlength{\baselineskip}{16.5pt}
%%%%%%%%%%%%%%%%%%%%%%%%%%%%%%%%%%%%%%%%%%%%%%%%%%%%%%%
%\section{Introduction}

%%%%%%%%%%%%%%%%%%%%%%%%%%%%%%%%%%%%%%%%%%%%%%%%%%%%%%%
\section*{Introduction}
It is well-known that secondary characteristic classes for foliations such as the Godbillon--Vey class vary under deformations of foliations~\cite{Thurston}, \cite{14}.
By taking derivatives of such classes with respect to differentiable one-parameter families, we can define characteristic classes for deformations of foliations~\cite{13}.
Beside these derivatives, given a deformation of a foliation, we can define a characteristic class which we call the Fuks--Lodder--Kotschick class (FLK class for short).
It is the most fundamental class which is \textit{not} the derivative of a secondary characteristic class.
Although the FLK class is non-trivial in the DGA associated with deformations of foliations~\cite{asuke:2015-2}, it seems unknown if there is a foliation with non-trivial FLK class in the real category.
Indeed, several conditions which assures the triviality of the FLK class are known~\cite{Kotschick}.
On the other hand, a non-trivial example is known in the transversely holomorphic setting (Example~\ref{ex3.2}), where the non-triviality is derived from the framings of the normal bundle.
In this article, we will show the following.
Relevant definitions will be given in Section~1.

\theoremstyle{plain}
\newtheorem*{TheoremA}{Theorem A}
\begin{TheoremA}
Let $\CF$ be a transversely projective foliation with trivial canonical bundle.
Then the FLK class of $\CF$ with respect to any infinitesimal deformation is trivial.
\end{TheoremA}

Theorem~A is a generalization of~\cite{Kotschick}*{Corollary~3.3}, where codimension-one foliations with transverse flat projective structures (called \textit{transverse homographic structures}) are considered.

\newtheorem*{TheoremB}{Theorem B}
\begin{TheoremB}
Let $\CF$ be a foliation with trivial canonical bundle.
Then the FLK class of $\CF$ is trivial for any infinitesimal deformations of $\CF$ if $H^1(M;\Theta_\CF)=\{0\}$ or $H^2(M;\Theta_\CF)=\{0\}$.
\end{TheoremB}

Theorems A and B are also valid for transversely holomorphic foliations, where the FLK class depends on the trivializations of the canonical bundle.
If a real foliations is considered, then we may assume that the canonical bundle is trivial by taking a double covering of $M$.
On the other hand, the triviality of the canonical bundle is an essential condition for transversely holomorphic foliations.

These theorems can be compared with the following ones concerning deformations of the Godbillon--Vey class.
\newtheorem*{TheoremA'}{Theorem A'}
\begin{TheoremA'}[\cite{Maszczyk},~\cite{asuke:tohoku},~\cite{asuke:2015},~\cite{asuke:2017}]
Let $\CF$ be a transversely projective foliation.
Then the derivative of the Godbillon--Vey class with respect to any infinitesimal deformation of $\CF$ is trivial.
\end{TheoremA'}

\newtheorem*{TheoremB'}{Theorem B'}
\begin{TheoremB'}[Heitsch~\cite{12},~\cite{13}, cf.~\cite{asuke:GV}]
Let $\CF$ be a foliation of $M$.
The derivative of the Godbillon--Vey class with respect to any infinitesimal deformation of $\CF$ is trivial if $H^1(M;\Theta_\CF)=\{0\}$.
\end{TheoremB'}

Theorems A' and B' are also valid for transversely holomorphic foliations if we consider the Bott class instead of the Godbillon--Vey class.

We also show the following

\newtheorem*{TheoremC}{Theorem C}
\begin{TheoremC}
The FLK class admits continuous variations in the category of transversely holomorphic foliations.
\end{TheoremC}

The definition of continuous variations is slightly subtle when compared with that of classical secondary classes such as the Godbillon--Vey class.
See Section~3 for details.

\section{Definitions}
Let $M$ be a manifold without boundary and $\CF$ a foliation of $M$.
We assume that foliations are regular, namely, without singularities.

We recall some basic definitions in order to fix notations.
% The following kind of foliations will appear.
\begin{definition}
A partition $\CF$ of $M$ into injectively immersed submanifolds $\{L_\lambda\}$ is called a \textit{foliation} of codimension $q$ if there is an atlas $\{U_i\}$, called a \textit{foliation atlas}, of $M$ such that each $U_i$ is homeomorphic to $V_i\x T_i$, where $V_i\subset\R^p$ is an open set and $T_i\subset\R^q$ is an open ball, in a way such that each connected component of $L_\lambda\cap U_i$ is of the form $V_i\x\{t\}$, where $t\in T_i$.
The local diffeomorphisms on $\disj T_i$ induced by parallel translations along the leaves are called holonomy.
\end{definition}
If there is an foliation atlas such that $\disj T_i$ admits an orientation which is preserved under the holonomy, then $\CF$ is said to be \textit{transversely orientable}.
In what follows, we assume that foliations are transversely orientable.
This is always the case if the foliation is transversely holomorphic.
In general, a structure on $\disj T_i$ invariant under the holonomy is called a transverse structure of the foliation $\CF$.
We might need a suitable choice of a foliation atlas in order to obtain a transverse structure depending on the structure.
In this article, the following one is relevant.

\begin{definition}
Let $\CF$ be a foliation of real codimension $2q$ of a manifold $M$.
The foliation $\CF$ is said to be \textit{transversely holomorphic} of complex codimension $q$ if there is a foliation atlas $\{U_i\}$ such that there exists an embedding of $\disj T_i$ into $\C^q$ in a way such that the holonomy consists of holomorphic mappings.
\end{definition}

Foliations without transverse holomorphic structures are also called \textit{real foliations} in contrast to transversely holomorphic ones.
In this article, we will study both real and transversely holomorphic foliations.
In what follows, $q$ will denote the codimension of $\CF$ if $\CF$ is a real foliation, and the complex codimension of $\CF$ if $\CF$ is transversely holomorphic.
On a foliation chart, say $U_i$, the coordinates in the transversal direction, namely, those for $T_i$, are denoted by $(y^1_i,\ldots,y^q_i)$ or simply by $(y^1,\ldots,y^q)$.
We set $K=\R$ for real foliations, and $K=\C$ for transversely holomorphic foliations.

\begin{definition}
If $\CF$ is a real foliation, we set $E(\CF)=T\CF$ and $Q(\CF)=TM/E(\CF)$.
We call $Q(\CF)$ the \textit{normal bundle} of $\CF$.
If $\CF$ is transversely holomorphic, we define $E(\CF)$ to be the complex vector bundle locally spanned by $T\CF$ and $\pdif{}{\overline{y_1}},\ldots\pdif{}{\overline{y_q}}$.
We set $Q(\CF)=(TM\otimes\C)/E(\CF)$ and call it the \textit{complex normal bundle} of $\CF$.
\end{definition}
Note that we always have $Q(\CF)=(TM\otimes K)/E(\CF)$.
In this article, normal bundles of transversely holomorphic foliations as real foliations will not appear so that $Q(\CF)$ is always referred as a normal bundle.

\begin{definition}
We set $K_\CF=\bigwedge^qQ^*(\CF)$ and call it the \textit{canonical bundle} of the foliation $\CF$.
\end{definition}

\begin{remark}
If we work on real foliations, then the triviality of $K_\CF$ is equivalent to the transversal orientability of $\CF$.
On the other hand, if we work on transversely holomorphic foliations, then the triviality of $K_\CF$ is equivalent to the triviality the first Chern class of $Q(\CF)$.
It can be attained by considering an $S^1$-fibration over $M$.
% This fibration is not used in this article so that we omit the details.
% The relationship of the lifted foliations and the original foliation $\CF$ is less simple than the real~case.
\end{remark}

The following transverse structure is also relevant.

\begin{definition}
A foliation is said to be \textit{transversely projective} if there is a foliation atlas $\{U_i\}$ such that $\disj T_i$ admits a projective structure which is invariant under the holonomy.
When transversely holomorphic foliations are considered, then projective structures are understood to be complex projective structures.
\end{definition}
Note that projective structures are not necessarily flat.
If the projective structure is flat, we may assume that $\disj T_j$ can be embedded in a projective space and that the holonomy pseudogroup consists of projective transformations.

\begin{definition}
Let $\Omega^r(U)=\vG_U(\bigwedge^rT^*M)$ be the set of $K$-valued differential forms of class $C^\infty$ on an open subset $U$ of $M$.
If $E$ is a vector bundle over $M$, we denote by $\Omega^r(U;E)=\vG_U(E\otimes\bigwedge^rT^*M)$ the set of $E$-valued $r$-forms on $U$.
We denote by $I^*_k(U;E)$ the ideal of $\Omega^*(U;E)$ locally generated by $s\otimes dy^{i_1}\wedge\cdots\wedge dy^{i_k}$, where $i_1<\cdots<i_k$ and $s\in\vG_U(E)$.
% If $E$ is a trivial line bundle, then we denote $I^r_k(U;E)$ by $I^r_k(U)$.
We set $C_\CF^r(U;Q(\CF))=\Omega^r(U;Q(\CF))/I^r_1(U;Q(\CF))$.
\end{definition}
We naturally have $C_\CF^r(U;Q(\CF))\cong\vG_U(Q(\CF)\otimes\bigwedge^rE(\CF)^*)$.

\begin{definition}
\label{def1.8}
A connection $\nabla^b$ on $Q(\CF)$ is called a \textit{Bott connection} if $\nabla^b_XY=\pi[X,\widetilde{Y}]$ holds for $X\in\vG_U(E(\CF))$ and $Y\in\vG_U(Q(\CF))$, where $U$ is an open set, $\pi\colon TM\otimes K\to Q(\CF)$ is the projection, $\widetilde{Y}$ is a lift of $Y$ to $TM\otimes K$.
A connection $\mathcal{D}$ on $K_\CF^{-1}$ is called a Bott connection if $\mathcal{D}_XY=L_XY$ holds for $X\in\vG_U(E(\CF))$ and $Y\in\vG_U(K_\CF^{-1})$, where $L_X$ denotes the Lie derivative with respect to $X$.
\end{definition}
Bott connections induced on related bundles such as $Q^*(\CF)$ are also called Bott connections.
Note that the induced connection on $K_\CF^{-1}$ by a Bott connection on $Q(\CF)$ is a Bott connection in the sense of Definition~\ref{def1.8}.

\begin{definition}
\label{d_F}
Let $\nabla^b$ be a Bott connection on $Q(\CF)$.
Let $\{e_1,\ldots,e_q\}$ be a local trivialization of $Q(\CF)$ and $\tau$ the connection form of $\nabla^b$ with respect to $\{e_1,\ldots,e_q\}$.
If $c\in C_\CF^r(U;Q(\CF))$, then we denote by $d_\CF$ the covariant exterior derivative, namely, we locally represent $c=\sum\limits_{i}e_i\otimes c^i$ and set
\[
d_\CF\,c=\sum_{i}e_i\otimes\left(dc^i+\sum_j\tau^i_j\wedge c^j\right)\mod I^{r+1}_1(U;Q(\CF)).
\]
We denote by $H^*_\CF(M;Q(\CF))$ the (co)homology of $(C_\CF^*(M;Q(\CF)),d_\CF)$.
\end{definition}
% Note that if we choose $\{dy^1,\ldots,dy^q\}$ as a local trivialization, then $\tau=\Gamma$.
It is known that $d_\CF$ is well-defined, namely, independent of the choice of Bott connections and local trivializations, and that $(C_\CF^*(M;Q(\CF)),d_\CF)$ is a cochain complex (cf.~\cite{12},~\cite{DuchampKalka},~\cite{asuke:2015}*{Lemma~3.4}).

\begin{definition}[\cite{14}]
A vector field $X$ is said to \textit{preserve} $\CF$ of a \textit{$\varGamma$ vector field} if $[X,Y]\in\vG(E(\CF))$ for any $Y\in\vG(E(\CF))$.
We denote by $\Theta_\CF$ the sheaf of germs of vector fields which preserve $\CF$.
\end{definition}
The following is known.

\begin{theorem}[Heitsch~\cite{12}, Duchamp--Kalka~\cite{DuchampKalka}]
\label{resolution}
The complex $(C_\CF^*(M;Q(\CF)),d_\CF)$ is a resolution of\/ $\Theta_\CF$ so that $H_\CF^*(M;Q(\CF))\cong H^*(M;\Theta_\CF)$.
\end{theorem}

\begin{definition}[Heitsch~\cite{12}, Duchamp--Kalka~\cite{DuchampKalka}]
Elements of $H^1(M;\Theta_\CF)$ are called \textit{infinitesimal deformations} of $\CF$.
\end{definition}

Let $\sigma\in H^1(M;\Theta_\CF)$ be an infinitesimal deformation of $\CF$.
We fix a representative, say $\dot{\omega}$, of $\sigma$ which is a $d_{\CF}$-closed one-form valued in $Q(\CF)$.
We choose a local trivialization $\{e_1,\ldots,e_q\}$ of $Q(\CF)$ and let $\omega$ be its dual.
Then, $\dot{\omega}$ is locally represented as $\dot{\omega}=\sum\limits_ie_i\otimes\dot{\omega}^i$, where $\dot{\omega}^i$ are one-forms.
We regard $\omega$ and $\dot{\omega}$ as $K^q$-valued one-forms.
If we denote by $\tau$ the connection form of $\nabla^b$ with respect to $\{e_1,\ldots,e_q\}$ as in Definition~\ref{d_F}, then, we can find a $\gl_q(K)$-valued one-form $\dot{\tau}$ such that
\[
d\dot{\omega}+\tau\wedge\dot{\omega}+\dot{\tau}\wedge\omega=0.
\]
If the local trivialization is changed from $e$ to $e'=eP$, where $P$ is a $\GL_q(K)$-valued function, then $\dot{\tau}$ is changed into $P^{-1}\dot{\tau}P$.
Therefore, if we set $\dot{\theta}=\tr\dot{\tau}$, then the one-form $\dot{\theta}$ is globally well-defined.
On the other hand, $\nabla^b$ induces a connection, say $\mathcal{D}$, on $K_\CF^{-1}=\bigwedge^qQ(\CF)$.
The connection form of $\mathcal{D}$ with respect to $e_1\wedge\cdots\wedge e_q$ is equal to $\tr\tau$.
Assume that $K_\CF^{-1}$ is trivial and fix a trivialization, say $e$.
Then the connection form of $\mathcal{D}$ with respect to $e$ is a globally well-defined one-form, which we denote by $\theta$.
Indeed, $\theta$ is given by the difference of the connection $\mathcal{D}$ and the flat connection with respect to $e$.
We denote $\theta$ by $\theta^e$ if we clarify $e$.
The one-form $\dot{\theta}$ can be regarded as an infinitesimal deformation of $\theta$ with respect to $\dot{\omega}$.
In addition, we have $d\theta=d\tr\tau$, which is independent~of~$e$.
% Even if $K_\CF^{-1}$ is non-trivial, the two-form $d\tr\tau$ is globally-well defined, which we denote also by $d\theta$.

\begin{definition}
Assume that $K_\CF^{-1}$ is trivial, and let $\theta$ and $d\theta$ be as~above.
We set $c=-1/2\pi$ for real foliations and $c=-1/2\pi\sqrt{-1}$ for transversely holomorphic foliations.
The class in $H^{2q+1}(M;K)$ represented by $c^{q+1}\theta\wedge(d\theta)^q$ is called the \textit{Godbillon--Vey class} of $\CF$ if $\CF$ is a real foliation, the \textit{Bott class} of $\CF$ if $\CF$ is a transversely holomorphic foliation.
\end{definition}

\begin{remark}
The Bott class can be defined as an element of $H^{2q+1}(M;\C/\Z)$ even if $K_\CF^{-1}$ is non-trivial.
\end{remark}

\begin{definition}
Let $\mu\in H^1(M;\Theta_\CF)$.
Let $\dot{\omega}$ be a representative of $\mu$ and we define $\theta$, $\dot{\theta}$, $d\theta$ as~above.
We set $c=-1/2\pi$ for real foliations and $c=-1/2\pi\sqrt{-1}$ for transversely holomorphic foliations.
\begin{enumerate}
\item
The class in $H^{2q+1}(M;K)$ represented by $c^{q+1}\dot{\theta}\wedge(d\theta)^q$ is called the \textit{infinitesimal derivative} with respect to $\mu$ of the Godbillon--Vey class if $\CF$ is real, of the Bott class if $\CF$ is transversely holomorphic.
We denote thus defined class by $\DGV_{\mu}(\CF)$ in the real case, $\DBott_{\mu}(\CF)$ in the transversely holomorphic case.
\item
Suppose that the canonical bundle $K_\CF$ of $\CF$ is trivial.
If $\CF$ is transversely holomorphic, then we fix the homotopy class of a trivialization, say $e$.
The class in $H^{2q+2}(M;K)$ represented by $c^{q+2}\dot{\theta}\wedge\theta\wedge(d\theta)^q$ is called the \textit{Fuks--Lodder--Kotschick class} (FLK class for short) with respect~to~$\mu$, and denoted by $\FLK_\mu(\CF)$.
If we emphasize the trivialization, then we denote $\FLK_\mu(\CF)$ also by $\FLK_{\mu}(\CF;e)$.
\end{enumerate}
\end{definition}
Note that $\mathrm{DBott}_\mu(\CF)$ is well-defined even if $K_\CF^{-1}$ is non-trivial.

The following is known.
\begin{theorem}[\cite{Fuks}*{Chapter~3, Section~1.5}, \cite{Lodder}, \cite{Kotschick}, \cite{asuke:GV}]
The FLK class is independent of the choices in the real case, while it depends on the homotopy class of trivializations of the canonical bundle in the transversely holomorphic case.
\end{theorem}

\section{Proofs of Theorems A and B}
The argument heavily depends on~\cite{asuke:2015}.
We will mostly follow the notations in~\cite{asuke:2015}*{Section~4}.
Let $\nabla^b$ be a Bott connection on $Q(\CF)$ and $\mathcal{D}$ the induced connection on $K_\CF^{-1}=\bigwedge^qQ(\CF)$.
% If $e$ is a trivialization of $K_\CF^{-1}$, then we denote by $\nabla^e$ the flat connection with respect to $e$.
If $e$ is a trivialization of $K_\CF^{-1}$, then we denote by $\theta^e$ the connection form of $\mathcal{D}$ with respect to $e$.
Note that $\theta^e$ is globally well-defined.
Note also that $d\theta^e$ is independent of the choice of $e$ once $\mathcal{D}$ is fixed, and moreover that $d\theta^e\in I_1^2(M;K)$.
We refer the reader to~\cite{BGJ:LNM279} for details.

\begin{lemma}
\label{lem2.1}
We have a well-defined mapping $\theta^e\wedge\colon H^r(M;\Theta_{\CF})\to H^{r+1}(M;\Theta_\CF)$ which depends on the homotopy class of the trivialization $e$ of $K_\CF^{-1}$ but not on $\mathcal{D}$.
\end{lemma}
\begin{proof}
Let $\{e_1,\ldots,e_q\}$ be a local trivialization of $Q(\CF)$ on an open set, say $U$.
% as in Definition~\ref{d_F}.
Let $\mu\in H^r(M;\Theta_\CF)$ and $c\in C_\CF^r(M;Q(\CF))$ a representative of $\mu$.
We represent $c$ as $c=\sum\limits_ie_i\otimes c^i$ on $U$.
As $c$ is a cocycle, by shrinking $U$ if necessary, there exists an element $\alpha=\sum_ie_i\otimes\alpha^i$ of $I_1^{r+1}(U;Q(\CF))$ such that
\[
\sum_ie_i\otimes\left(dc^i+\sum_j\tau^i_j\wedge c^j\right)=\sum_ie_i\otimes\alpha^i.
\]
We have
\begin{align*}
&\hphantom{{}={}}%
d_\CF(\theta^e\wedge c)\\*
&=\sum_ie_i\otimes\left(d(\theta^e\wedge c^i)+\sum_j\tau^i_j\wedge(\theta^e\wedge c^j)\right)\\*
&=\sum_ie_i\otimes\left(d\theta^e\wedge c^i-\theta^e\wedge dc^i+\sum_j\tau^i_j\wedge(\theta^e\wedge c^j)\right)\\*
&=\sum_ie_i\otimes\left(d\theta^e\wedge c^i-\theta^e\wedge\alpha^i\right).
\end{align*}
As $d\theta^e\in I_1^2(U;K)$, $\theta^e\wedge c$ is $d_\CF$-closed.
Suppose moreover that $c=d_\CF f$ holds for some $f\in C_\CF^{r-1}(M;Q(\CF))$.
We locally represent $f$ as $f=\sum\limits_ie_i\otimes f^i$.
Then we have
\[
c^i=\sum_ie_i\otimes\left(df^i+\sum_j\tau^i_j\wedge f^j\right)+\sum_ie_i\otimes\alpha^i
\]
for some $\{\alpha^1,\ldots,\alpha^q\}\in I_1^r(U;K)$.
We have
\begin{align*}
&\hphantom{{}={}}\theta^e\wedge c\\*
&=\sum_ie_i\otimes\left(\theta^e\wedge df^i+\sum_j\theta^e\wedge\tau^i_j\wedge f^j\right)+\sum_ie_i\otimes\theta^e\wedge\alpha^i\\*
&=-\sum_ie_i\otimes\left(d(\theta^e\wedge f^i)+\sum_j\tau^i_j\wedge(\theta^e\wedge f^j)\right)+\sum_ie_i\otimes d\theta^e\wedge f^i+\sum_ie_i\otimes\theta^e\wedge\alpha^i.
\end{align*}
Note that $\{\sum_ie_i\otimes\theta^e\wedge f^i\}$ gives rise to an element $\theta^e\wedge f\in C^r(M;Q(\CF))$.
Again as $d\theta^e\in I_1^2(U;K)$, we have $\theta^e\wedge c=d_\CF(-\theta^e\wedge f)$.
Thus $\theta^e\wedge\mu$ is well-defined.
Let $\mathcal{D}'$ also be a Bott connection on $K_\CF^{-1}$ and $\theta'{}^e$ the connection form of $\mathcal{D}'$ with respect to $e$.
Then, $\theta'{}^e-\theta^e\in I^1_1(M;Q(\CF))$ so that we have $\theta^e\wedge=\theta'{}^e\wedge$.
Finally we replace the trivialization of $K_\CF^{-1}$ by $\widehat{e}=fe$, where $f$ is a nowhere zero function.
If we assume that $\widehat{e}$ is homotopic to $e$, then we can choose a well-defined branch $\log f$ of the logarithm of $f$.
On the other hand, the connection form of $\mathcal{D}$ with respect to $\widehat{e}$ is given by $\theta^e+\dfrac{df}{f}=\theta^e+d\log f$.
We locally have
\begin{align*}
d((\log f)c)+\tau\wedge((\log f)c)%
&=(d\log f)\wedge c+(\log f)(dc+\tau\wedge c)\\*
&=(d\log f)\wedge c+(\log f)\alpha,
\end{align*}
where $\alpha\in I_1^{r+1}(U;Q(\CF))$.
Therefore, $\theta^e\wedge\mu\in H^{r+1}(M;\Theta_\CF)$ is independent of the homotopy class of the trivialization of~$K_\CF^{-1}$.
\end{proof}

\begin{remark}
It suffices to assume that the connection $\mathcal{D}$ is a Bott connection, not necessarily induced by $\nabla^b$ in Lemma~\ref{lem2.1}.
\end{remark}

\begin{remark}
The homotopy type of the trivialization of $K_\CF^{-1}$ is unique in the real case once the transverse orientation is fixed.
\end{remark}

\begin{remark}
We retain the notations in the proof Lemma~\ref{lem2.1}.
Let $\{U_i\}$ be a locally finite simple covering of $M$ such that each $U_i$ is contained in a foliation chart.
If $e$ is a trivialization of $K_\CF^{-1}$, then $e$ is represented as $f_i\pdif{}{y^1_i}\wedge\cdots\wedge\pdif{}{y^q_i}$ on $U_i$, where $f_i$ is a non-vanishing function on $U_i$.
If we denote by $\theta^e_i$ the restriction of $\theta^e$ to $U_i$, then $\theta^e_i-d\log f_i\in I_1^1(U_i;Q(\CF))$, where we fix a branch of $\log f_i$ if $\CF$ is transversely holomorphic.
Therefore, $\theta^e\wedge c=d\log f_i\wedge c$ holds on $U_i$ modulo $I_1^r(U_i;Q(\CF))$.
On the other hand, we have
\[
d_\CF(\log f_ic)=d\log f_i\wedge c+\theta\wedge(\log f_ic)+\log f_idc=d\log f_i\wedge c
\]
on $U_i$ modulo $I_1^{r+1}(U_i;Q(\CF))$.
Continuing in this way, we can show that if $\mu\in H^1(M;\Theta_\CF)$ is represented by a family $\{X_{ij}\}$ of foliated vector fields, namely, each $X_{ij}$ is a vector field on $U_i\cap U_j$ locally constant along the leaves and in addition transversely holomorphic if $\CF$ is so, then there is a family $\{m_{ij}\}$ of integers, which reflects the homotopy type of $e$, such that $\theta^e\wedge\mu$ is represented by $(\log\det D\gamma_{ij}+2\pi\sqrt{-1}m_{ij})X_{jk}$ up to signature depending on conventions, where $(y^1_j,\ldots,y^q_j)=\gamma_{ji}(y^1_i,\ldots,y^q_i)$ and we fix a branch of $\log\det D\gamma_{ij}$.
Note that we may assume that $m_{ij}=0$ if $\CF$ is a real foliation.
This shows that $\theta^e\wedge\mu$ is different from the primary obstruction of Kodaira~\cite{Kodaira} for deformations.
Indeed, Proposition~\ref{prop3.1} and Example~\ref{ex3.2} show that there is a family $\{\CF_\lambda\}$ of transversely holomorphic foliations such that $\mathrm{DBott}_\mu(\CF_\lambda)$ and $\FLK_{\mu}(\CF_\lambda;e_\lambda)$ varies together under variation of $\lambda$.
In particular, $\theta^e\wedge\mu$ should be non-trivial.
Hence $\theta^e\wedge\mu$ does not correspond to the obstruction for $\mu$.
\end{remark}

The following is known.
\begin{proposition}[\cite{asuke:2015}*{Propositions~4.7,~4.12}]
\label{prop3.4}
There is a well-defined mapping $\mathcal{L}_P\colon H_\CF^r(M;Q(\CF))\to H^{2q+r}(M)$ which depends on the equivalent class of transversely torsion-free Bott connections on $Q(\CF)$.
If $\CF$ is transversely projective, then $\mathcal{L}_P=0$.
\end{proposition}

The mapping $\mathcal{L}_P$ is given as follows.
Let $U$ be a foliation chart and $y=(y^1,\ldots,y^q)$ the coordinates in the transversal direction.
If we choose $\pdif{}{y^1}\wedge\ldots\wedge\pdif{}{y^q}$ as a local trivialization of $K_{\CF}^{-1}$, then, $\mathcal{D}$ is locally represented by $\theta=f_1dy^1+\cdots+f_qdy^q$, where $f_1,\ldots,f_q$ are functions on $U$.
Let
\[
N_i=df_i-\sum_j\frac1{q+1}f_if_jdy^j.
\]
Then $\mathcal{L}_P(\mu)$, where $\mu=[c]$, is represented by
\[
d(N\wedge c)\wedge(d\theta)^{q-1}
\]
which is a globally well-defined $(2q+r)$-form~\cite{asuke:2015}*{Definition~4.6, Proposition~4.7}.

The mapping $\mathcal{L}_P$ is derived from a kind of Cartan connection called a transverse TW-connection, where `TW' stands for `Thomas--Whitehead' (cf.~\cite{Roberts}).
A transverse TW-connection is a linear connection on $Q(\widetilde{\CF})$, where $\widetilde{\CF}$ is a lift of $\CF$ to the associated principal bundle with $K_\CF^{-1}$ (see~\cite{asuke:2015}, also~\cite{Kobayashi}).
We can show the following

\begin{theorem}[\cite{asuke:2015}*{Definition~2.1 and Theorem~2.3}]
\label{thm1.5}
A foliation is transversely projective if there is a transverse TW-connection on $Q(\widetilde{\CF})$ invariant under the holonomy.
\end{theorem}

\begin{lemma}[cf.~\cite{asuke:2015}*{Lemma~4.10, Theorem~4.11}]
\label{thm3.3}
We have
\[
\mathcal{L}_P(\theta^e\wedge[\dot{\omega}])=\frac1q\dot{\theta}\wedge\theta^e\wedge(d\theta)^q.
\]
\end{lemma}
\begin{proof}
The class $\mathcal{L}_P(\theta^e\wedge[\dot{\omega}])$ is represented by $d(N\wedge\theta^e\wedge\dot{\omega})\wedge(d\theta)^{q-1}$.
We~have
\begin{align*}
d(N\wedge\theta^e\wedge\dot{\omega})\wedge(d\theta)^{q-1}%
&=-d\theta\wedge(N\wedge\dot{\omega})\wedge(d\theta)^{q-1}+\theta^e\wedge d(N\wedge\dot{\omega})\wedge(d\theta^{q-1})\\*
&=-\left(\sum_idf_i\wedge\dot{\omega}^i\right)\wedge(d\theta)^q-\frac1q\theta^e\wedge\dot{\theta}\wedge(d\theta)^q\\*
&=-\frac1q\theta^e\wedge\dot{\theta}\wedge(d\theta)^q,
\end{align*}
because we have $(d\theta)^q=q!df_1\wedge\cdots\wedge df_q\wedge dy^1\wedge\cdots\wedge dy^q$.
\end{proof}

Theorem A now follows from Proposition~\ref{prop3.4} and Lemma~\ref{thm3.3}.
Theorem B follows from Lemma~\ref{thm3.3}, because $\theta^e\wedge[\dot{\omega}]$ should be trivial for any $e$.

\section{Non-triviality of the FLK class}
It is difficult to find an example of a real foliation with non-trivial FLK class~\cite{Kotschick}.
On the other hand, there is a following simple construction to yield non-trivial examples in the transversely holomorphic setting.
Let $\CF$ be a transversely holomorphic foliation of complex codimension $q$.
Suppose that the canonical bundle $K_\CF$ is trivial and let $\omega$ be a trivialization.
We denote by $e$ the trivialization of $K_\CF^{-1}$ dual to $\omega$.
Note that $\omega$ defines $\CF$ in the sense that we have $E(\CF)=\ker\omega=\{X\in TM\otimes\C\mid\iota_X\omega=0\}$.
Let $M^\circ=M\x S^1$, where $S^1$ is considered as the unit circle in $\C$ and $t$ will denote the natural coordinates.
We denote by $\pi\colon M\x S^1\to M$ the projection.
Let $\CF^\circ$ be the pull-back of $\CF$ to $M^\circ$ of which the leaves are of the form $L\x S^1$, where $L$ is a leaf of $\CF$.
Then, the canonical bundle of $\CF^\circ$ is also trivial.
Indeed, we can consider a trivialization $t^m\omega$ of $K_{\CF^\circ}$, where $m\in\Z$.
We denote by $e_m$ the trivialization of $K_{\CF^\circ}^{-1}$ dual to $t^m\omega$.
As $\omega$ determines $\CF$, we have a $\C$-valued one-form $\eta$ such that $d\omega=-\eta\wedge\omega$.
If we~set
\[
\eta^\circ=\eta-m\frac{dt}{t},
\]
then we have $d(t^m\omega)=-\eta^\circ\wedge(t^m\omega)$, where we omit $\pi^*$ if it is apparent.
There is a Bott connection on $K_\CF^{-1}$ of which the connection form with respect to $e$ is equal to $\eta$.
Similarly, there is a Bott connection on $K_{\CF^\circ}^{-1}$ of which the connection form with respect to $e_m$ is equal to $\eta^\circ$.
Suppose that $\CF$ admits a differentiable one-parameter family $\{\CF_\mu\}$ with $\CF_0=\CF$, where $0$ is the base point.
If we represent the derivatives at $\mu=0$ by adding dots, then $(-1/2\pi\sqrt{-1})^{q+1}\dot{\eta^\circ}\wedge(d\eta^\circ)^q$ and $(-1/2\pi\sqrt{-1})^{q+2}\dot{\eta^\circ}\wedge\eta^\circ\wedge(d\eta^\circ)^q$ represent $\DBott_\mu(\CF^\circ)$ and $\FLK_\mu(\CF^\circ;e_m)$, respectively.
If $\FLK_\mu(\CF;e)$ is non-trivial, then the foliation is what we look for.
On the other hand, we have the following proposition which generalizes~\cite{asuke:GV}*{Example~5.11}.

\begin{proposition}
\label{prop3.1}
Suppose that $\FLK_\mu(\CF;e)$ is trivial in $H^{2q+2}(M;\C)$.
Suppose in addition that $\DBott_\mu(\CF)$ is non-trivial in $H^{2q+1}(M;\C)$.
Then, the FLK class $\FLK_{\mu^\circ}(\CF^\circ;e_m)$ is non-trivial in $H^{2q+2}(M\x S^1;\C)$ if $m\neq0$, where $\mu^\circ=\pi^*\mu$.
More precisely, we have $\pi_!\FLK_{\mu^\circ}(\CF^\circ;e_m)=-m\DBott_\mu(\CF)$, where $\pi_!$ denotes the integration along the fiber.
In particular, $H^2(M\x S^1;\Theta_{\CF^\circ})\neq\{0\}$.
\end{proposition}
\begin{proof}
We have
\begin{align*}
\dot{\eta^\circ}\wedge\eta^\circ\wedge(d\eta^\circ)^q%
&=\dot{\eta}\wedge\left(\eta-m\frac{dt}{t}\right)\wedge(d\eta)^q\\*
&=\dot{\eta}\wedge\eta\wedge(d\eta)^q+m\frac{dt}{t}\wedge\dot{\eta}\wedge(d\eta)^q.
\end{align*}
By multiplying $(-1/2\pi\sqrt{-1})^{q+2}$ on the both hand sides, we obtain
\[
\FLK_{\mu^\circ}(\CF^\circ;e_m)=\pi^*\FLK_{\mu}(\CF;e)-m\pi^*\DBott_\mu(\CF)\wedge\mathrm{vol}_{S^1}.
\]
The first part of the proposition follows from this equality.
The second claim follows from Theorem~B.
\end{proof}

\begin{example}[\cite{asuke:GV}*{Example~5.11}, see also~\cite{BGJ:LNM279}*{pp.~74--76}]
\label{ex3.2}
Let $\CF_\lambda$ be a foliation of $S^3$ given by the one-form
\[
\lambda_2z_2dz_1-\lambda_1z_1dz_2,
\]
where $S^3$ is considered to be the unit sphere in $\C^2$ of which the coordinates are given by $(z_1,z_2)$, and $\lambda_1,\lambda_2$ are non-zero complex numbers such that the ratio $\lambda=\lambda_1/\lambda_2$ is not a negative real number.
It is well-known that $\mathrm{Bott}(\CF_\lambda)=\left(\lambda+2+\frac1{\lambda}\right)[S^3]$, where $[S^3]$ denotes the standard generator of $H^3(S^3;\C)$.
On the other hand, the FLK class is trivial regardless trivializations of $Q(\CF)$ because it is of degree $4$.
If we denote by $\mu$ the element of $H^1(M;\Theta_{\mathcal{F_\lambda}})$ which corresponds to the variation of $\lambda$, then $\DBott_{\mu}(\CF_\lambda)=\left(1-\frac1{\lambda^2}\right)[S^3]$.
Let
\[
\omega_m=t^m(\lambda_2z_2\,dz_1-\lambda_1z_1\,dz_2)
\]
be a trivialization of $K_{\CF_\lambda^\circ}=Q^*(\CF_\lambda^\circ)$, where $t$ denotes the standard coordinates of $S^1$ considered as the unit circle in $\C$.
We denote $e_m$ the trivialization of $K_{\CF_\lambda^\circ}^{-1}$ dual to $\omega_m$.
Then, by Proposition~\ref{prop3.1}, $\FLK_{\mu}(\CF_\lambda^\circ;e_m)=-m\pi^*\DBott_{\mu}(\CF_\lambda)\cup[S^1]$, where $[S^1]$ denotes the image of the standard generator of $H^1(S^1;\C)$ in $H^*(S^3\x S^1;\C)$.
It is easy to generalize this construction to obtain an example of arbitrary complex codimension.
\end{example}

\begin{remark}
The FLK classes for examples of Heitsch~\cite{14} can be shown trivial.
For example, in Example~1 of~\cite{14}, real foliation of codimension $2q-1$ are constructed on fiber bundles $F\to E\to B$ with $B$ being a product of a $q$-tuple of surfaces and $F$ the $2q$-dimensional Euclidean space equipped with vector fields of the form $\lambda_1y^1\pdif{}{y^1}+\cdots+\lambda_{2q}y^{2q}\pdif{}{y^{2q}}$, where $\lambda_{2i-1}=\lambda_{2i}$.
The foliations are spanned by a kind of lifts of $B$ and the vector fields, and are essentially on the $S^{2q-1}$-bundle over $B$.
The Godbillon--Vey class of these foliations are of degree $4q-1$ and highly non-trivial, while classes of degree $4q$ should be trivial.
Example~2 of~\cite{14} also yields a trivial example, although computations become involved.
\end{remark}

We can consider continuous deformations of the FLK class.
We should be aware that $H^1(M;\Theta_{\CF_\lambda})$ is not necessarily constant.
We propose the following.

\begin{definition}
Let $\{\CF_\lambda\}$ be a differentiable family of foliations of a fixed manifold $M$ with parameter $\lambda$.
We assume that the codimension of $\CF_\lambda$ is constant and that transverse holomorphic structure also varies differentiably if each $\CF_\lambda$ is transversely holomorphic.
Assume in addition that there exist a differentiable family of $1$-forms $\dot\omega_\lambda$ such that each $\dot\omega_\lambda$ is $d_{\CF_\lambda}$-closed, and a differentiable family $\{\widetilde{e}_\lambda\}$ of sections to $\bigwedge^qTM\otimes K$ such that $\widetilde{e}_\lambda$ induces a trivialization, say $e_\lambda$, of $K_{\CF_\lambda}^{-1}$ for each $\lambda$.
If $\FLK_{\mu_\lambda}(\CF_\lambda;e_\lambda)$ varies continuously in $H^{2q+2}(M;K)$ which is a vector space, then we say that the FLK class admits a \textit{continuous variation}.
\end{definition}

Theorem~C in the introduction now follows from Example~\ref{ex3.2}.
We do not know if Theorem~C is valid for real foliations.

\end{document}